\documentclass[11pt, reqno]{article}
\usepackage[T1]{fontenc}
\usepackage{amsmath,amssymb,amsfonts,amsthm}
\usepackage{color}

\usepackage[colorlinks=true,linkcolor=blue,citecolor=blue,urlcolor=blue,pdfborder={0 0 0}]{hyperref}
\usepackage{cleveref}

\usepackage{caption}
\usepackage{thmtools}

\usepackage{mathrsfs}

\crefname{theorem}{Theorem}{Theorems}
\crefname{thm}{Theorem}{Theorems}
\crefname{lemma}{Lemma}{Lemmas}
\crefname{lem}{Lemma}{Lemmas}
\crefname{remark}{Remark}{Remarks}
\crefname{prop}{Proposition}{Propositions}
\crefname{defn}{Definition}{Definitions}
\crefname{corollary}{Corollary}{Corollaries}
\crefname{conjecture}{Conjecture}{Conjectures}
\crefname{question}{Question}{Questions}
\crefname{chapter}{Chapter}{Chapters}
\crefname{section}{Section}{Sections}
\crefname{figure}{Figure}{Figures}
\crefname{example}{Example}{Examples}

\theoremstyle{plain}
\newtheorem{thm}{Theorem}[section]

\newtheorem{lemma}[thm]{Lemma}
\newtheorem{theorem}[thm]{Theorem}

\newtheorem{corollary}[thm]{Corollary}

\newtheorem{prop}[thm]{Proposition}

\theoremstyle{definition}

\newtheorem{example}[thm]{Example}

\theoremstyle{remark}
\newtheorem{remark}[thm]{Remark}

\numberwithin{equation}{section}

\renewcommand{\P}{\mathbb P}
\newcommand{\E}{\mathbb E}

\newcommand{\Z}{\mathbb Z}

\newcommand{\Q}{\mathbb Q}

\usepackage[margin=1in]{geometry}
\usepackage{tikz}
\usepackage{pgfplots}
\usepackage{extarrows}
\usepackage{titling}
\usepackage{commath}
\usepackage{wasysym}
\usepackage{mathtools}
\usepackage{titlesec}
\newcommand{\eps}{\varepsilon}
\usepackage{bbm}
\usepackage{setspace}
\usepackage{enumitem}
\usepackage{booktabs}

\usepackage[textsize=tiny]{todonotes}

\usepackage{cite}

\newcommand{\Vol}{\operatorname{Vol}}

\newcommand{\tv}{\operatorname{TV}}

\def\P{\mathbb{P}}

\DeclareMathSymbol{\leqslant}{\mathalpha}{AMSa}{"36} 
\DeclareMathSymbol{\geqslant}{\mathalpha}{AMSa}{"3E} 
\DeclareMathSymbol{\eset}{\mathalpha}{AMSb}{"3F}     

\renewcommand{\epsilon}{\varepsilon}

\makeatletter
\pgfdeclareshape{slash underlined}
{
  \inheritsavedanchors[from=rectangle] 
  \inheritanchorborder[from=rectangle]
  \inheritanchor[from=rectangle]{north}
  \inheritanchor[from=rectangle]{north west}
  \inheritanchor[from=rectangle]{north east}
  \inheritanchor[from=rectangle]{center}
  \inheritanchor[from=rectangle]{west}
  \inheritanchor[from=rectangle]{east}
  \inheritanchor[from=rectangle]{mid}
  \inheritanchor[from=rectangle]{mid west}
  \inheritanchor[from=rectangle]{mid east}
  \inheritanchor[from=rectangle]{base}
  \inheritanchor[from=rectangle]{base west}
  \inheritanchor[from=rectangle]{base east}
  \inheritanchor[from=rectangle]{south}
  \inheritanchor[from=rectangle]{south west}
  \inheritanchor[from=rectangle]{south east}
  \inheritanchorborder[from=rectangle]
  \foregroundpath{
    \southwest \pgf@xa=\pgf@x \pgf@ya=\pgf@y
    \northeast \pgf@xb=\pgf@x \pgf@yb=\pgf@y
    \pgf@xc=\pgf@xa
    \advance\pgf@xc by .5\pgf@xb
    \pgf@yc=\pgf@ya
    \advance\pgf@xc by -1.3pt
    \advance\pgf@yc by -1.8pt
    \pgfpathmoveto{\pgfqpoint{\pgf@xc}{\pgf@yc}}
    \advance\pgf@xc by  2.6pt
    \advance\pgf@yc by  3.6pt
    \pgfpathlineto{\pgfqpoint{\pgf@xc}{\pgf@yc}}
    \pgfpathmoveto{\pgfqpoint{\pgf@xa}{\pgf@ya}}
    \pgfpathlineto{\pgfqpoint{\pgf@xb}{\pgf@ya}}
 }
}
\tikzset{nomorepostaction/.code=\let\tikz@postactions\pgfutil@empty}

\makeatother

\pretitle{\begin{flushleft}\Large}
\posttitle{\par\end{flushleft}\vskip 0.5em
}
\preauthor{\begin{flushleft}}
\postauthor{
\\
\vspace{0.5em}
\footnotesize{$*$ The Division of Physics, Mathematics and Astronomy, California Institute of Technology}
\\
Email:
\href{mailto:t.hutchcroft@caltech.edu}{t.hutchcroft@caltech.edu}\\
$\dagger$ Department of Mathematics, Massachusetts Institute of Technology\\
Email: \href{mailto:imlopez@mit.edu}{imlopez@mit.edu}
\end{flushleft}}
\predate{\begin{flushleft}}
\postdate{\par\end{flushleft}}
\title{{\bf A relation between isoperimetry and total variation decay 
with applications to graphs of non-negative Ollivier-Ricci curvature}}

\renewenvironment{abstract}
 {\par\noindent\textbf{\abstractname.}\ \ignorespaces}
 {\par\medskip}

\titleformat{\section}
  {\normalfont\large \bf}{\thesection}{0.4em}{}

\author{{\bf Tom Hutchcroft$^*$ and Isaac M.\ Lopez$^\dagger$}}

\begin{document}

\maketitle

\begin{abstract}
We prove an inequality relating the isoperimetric profile of a graph to the decay of the random walk total variation distance $\sup_{x\sim y} \|P^n(x,\cdot)-P^n(y,\cdot)\|_\mathrm{TV}$.
This inequality implies a quantitative version of a theorem of Salez (GAFA 2022) stating that bounded-degree graphs of non-negative Ollivier--Ricci curvature cannot be expanders. 
Along the way, we prove universal upper-tail estimates for the random walk displacement $d(X_0,X_n)$ and information $-\log P^n(X_0,X_n)$, which may be of independent interest.
\end{abstract}



\section{Introduction}

There is a long and fruitful history of inequalities relating various geometric and spectral features of graphs and groups: see \cite{zheng2022asymptotic} for a survey and e.g.\ \cite{LP:book,KumagaiBook} for more detailed introductions. In this note, we prove a new inequality stating roughly that if the total variation distance between the random walks started at any two neighbors of a graph is small at some time $n$, then the graph must contain sets of reasonable size with small boundary-to-volume ratio in a quantitative sense. This new inequality allows us to make a recent result of Salez \cite{salez2022sparse} on the non-existence of expander sequences of non-negative Ollivier-Ricci curvature fully quantitative, as we explain in \cref{subsec:Ricci}.


\medskip

Let us now give the definitions required to state our theorem. 
    Let $G=(V,E)$ be a connected, locally finite graph, and let $d_G$ denote the graph distance on $G$. Given a set of vertices $W \subseteq V$, we write $\partial W$ for the set of edges with one endpoint in $W$ and the other not in $W$. We write $P^n$ for the $n$-step transition matrix of the lazy simple random walk on $G$ (which stays still with probability $1/2$ at each step) and define the \textbf{total variation profile} $\mathrm{TV}:\mathbb{N} \to [0,\infty)$ of $G$ by
    \begin{align*}
        \mathrm{TV}_n=\mathrm{TV}_{G,n}:= \sup_{x\sim y}\|P^n(x,\cdot)-P^n(y,\cdot)\|_\mathrm{TV} = \sup_{x\sim y} \frac{1}{2}\sum_{z\in V}|P^n(x,z)-P^n(y,z)|,
    \end{align*}
    where $P^n(x,\cdot)$ denotes the law of the $n$th step of a lazy random walk started at $x$.  It is a standard fact that $\|P^n(x,\cdot)-P^n(y,\cdot)\|_\mathrm{TV}=\inf \P(X_n \neq Y_n)$, where the infimum is taken over couplings of the lazy random walks $X$ and $Y$ started at $x$ and $y$, respectively. 
    It is easily seen that $\mathrm{TV}_n$ is monotonically non-increasing in $n$. For transitive graphs, $\operatorname{TV}_n\to 0$ as $n\to\infty$ if and only if $G$ has the \emph{Liouville property}, meaning that there are no non-constant bounded harmonic functions on $G$.
We also define 
\begin{align*}
D^*_n=D^*_{G,n}&:= \sup_{x \in V}\max_{0\leq m \leq n} \mathbf{E}_x [d_G(X_0,X_m)] \qquad \text{ and}\\ H^*_n=H^*_{G,n}&:= \sup_{x\in V} \max_{0\leq m \leq n}  \sum_{y} P^m(x,y) (-\log P^m(x,y)) 
\end{align*}
to be the supremal expected displacement and entropy of the lazy random walk on $G$ run for $n$ steps, respectively, where we write $\mathbf{P}_x$ and $\mathbf{E}_x$ for probabilities and expectations taken with respect to the lazy random walk started at $x$.

\medskip

 Recall that a graph is said to be \textbf{amenable} if $\inf\{|\partial W|/\sum_{w\in W}\deg(w):W\subseteq V$ finite$\}=0$; finite graphs are trivially amenable since one can take $W=V$, which has $\partial W=\emptyset$. It is proven in \cite{salez2022sparse} that every bounded-degree graph with $\inf_n \operatorname{TV}_n=0$ is amenable.  The main result of this paper makes this quantitative.

\begin{theorem}
\label{thm:main}
For each $M<\infty$, there exists a finite constant $C(M)$ such that the following holds.
Let $G=(V,E)$ be an amenable connected graph with degrees bounded by $M$. For each $n\geq 1$, there exists a set $W_n \subseteq V$ with diameter and volume satisfying \[
\operatorname{diam}(W_n) \leq \min\Bigl\{2n, C D_n^* \log \frac{1}{\mathrm{TV}_n}\Bigr\} \quad \text{ and } \quad
\log |W_n| \leq \min\Bigl\{2Mn,C (H^*_n+\log n)\log \frac{1}{\mathrm{TV}_n}\Bigr\}\]
 and whose boundary-to-volume ratio is bounded by $4\mathrm{TV}_n$:  $|\partial W_n|/\sum_{w\in W}\deg(w)\leq 4 \mathrm{TV}_n$.
\end{theorem}

Informally, this theorem states that if $\operatorname{TV}_n$ is small, then we can find sets of small boundary-to-volume ratio that are not too large, and that the size constraints on these sets can be improved if we have additional information on the random walk's displacement or entropy. The inequality is most interesting in situations where we have better understanding of $\operatorname{TV}_n$ than any other geometric or spectral properties of the graph, as is currently the case for graphs of non-negative Ollivier-Ricci curvature (\cref{subsec:Ricci}). Note that for graphs of small growth, one often gets a better bound on the size of the set using the diameter condition than from the bound involving the entropy.

\begin{remark}
The proof of Theorem~\ref{thm:main} yields an explicit method to construct the sets $W_n$ using a simultaneous near-optimal coupling of the lazy random walks started at every vertex of $G$.
For finite graphs, the proof shows moreover that for each $n\geq 1$, the graph can be \textit{partitioned} into sets of diameter and volume bounded as in the statement of the theorem such that, say, $(100-\eps)\%$ of vertices belong to a cell with boundary-to-volume ratio at most $O(\eps^{-1}\mathrm{TV}_n)$. (For infinite amenable graphs, a similar statement holds with the density of good vertices measured via a F{\o}lner sequence.) In fact, the construction of this partition can be carried out using a randomized algorithm that respects all symmetries of the graph, which can be implemented using a factor-of-iid construction on infinite Cayley graphs and other unimodular random rooted graphs; we do not pursue this here. (The proof is inspired by proofs that various notions of amenability and hyperfiniteness are equivalent for measured equivalence relations \cite[Section 8]{AL07}, where such factor-of-iid considerations can be important.)
\end{remark}

\medskip

Theorem \ref{thm:main} has the following elementary analytic consequence. We state this corollary in terms of the \textbf{isoperimetric profile}
\[
\Phi(n) =\Phi(G,n) := \inf\left\{\frac{|\partial W|}{\sum_{w\in W} \deg(w)} : \sum_{w\in W} \deg(w)\leq n\right\},
\]
which measures how small the boundary-to-volume ratio of a set can be as a function of its volume.

\begin{corollary} \label{isoperimetric-estimate-from-entropy-and-TV}
Let $G$ be an infinite bounded degree graph. If $H^*_n=O(n^\alpha)$ and $\mathrm{TV}_n=O(n^{-\beta})$ for some $\alpha,\beta \in (0,1]$, then the isoperimetric profile of $G$ satisfies
\[
\Phi(n) = O\left( \left(\frac{\log \log n}{\log n}\right)^{\beta/\alpha}\right)
\]
as $n\to\infty$.
\end{corollary}

\begin{example}
We now give some examples to demonstrate the near-optimality of
Theorem~\ref{thm:main}. For $\Z^d$, we have that $D_n^* \asymp n^{1/2}$ and $\operatorname{TV}_n \asymp n^{-1/2}$, and since the ball of radius $r$ contains $r^d$ points, it follows from Theorem~\ref{thm:main} that for every $n\geq 1$, there exists a set of volume at most $n^{d/2} (\log n)^d$ and boundary-to-volume ratio at most $n^{-1/2}$ so that
\[
\Phi(n) =O\left(\frac{\log n}{n^{1/d}}\right);
\]
this bound is larger than the true order $\Phi(n)\asymp n^{-1/d}$ by a $\log n$ factor. 
For nearest-neighbor random walk on the lamplighter group $\Z_2 \wr \Z$, we have that $H^*_n \asymp n^{1/2}$ 
and\footnote{Here is a brief explanation of how to bound $\mathrm{TV}_n$ on the lamplighter group. We can couple the walks started at two different points by first waiting for the two associated walks on $\Z$ to couple at some time $T$, and then waiting for the subsequent single walk on $\Z$ to visit every point in the range of the two walks up to this coupling time, making sure to set the two lamp configurations to be equal wherever the two walks go after the coupling time. If the union of the two ranges has diameter $R$ at time $T$, then the probability that the this range has not yet been covered at time $T+n$ is $1-(1-O(n^{-1/2}))^R$. This last fact is most easily seen by replacing $n$ by a geometric random time of mean $n$ and using the strong Markov property at successive running minima of the walk started from time $T$. Taking a union bound over whether or not $T\leq n/2$, the total probability that the two walks on the lamplighter have failed to couple by time $n$ is of order at most $n^{-1/2}+1-\E[(1-Cn^{-1/2})^{R'}]$ where $R'$ is the range of the two random walks at time $T\wedge (n/2)$. Using Jensen and considering each walk separately we can bound the second term by something of order $n^{-1/2}\E[\max_{0\leq m \leq T \wedge n} |X_m|]$. This expectation can be bounded by writing $\E[\max_{0\leq m \leq T \wedge n} |X_m|]\leq \sum_{k=1}^{\log_2 n} \E[\max_{0\leq m \leq 2^k} |X_m| \mathbbm{1}(T \geq 2^{k-1})]$, and it follows from Cauchy-Schwarz and Doob's $L^2$ maximal inequality that each term in this sum is $O(1)$. This shows that the total variation distance is at most $n^{-1/2}\log n$ as claimed.
} $\mathrm{TV}_n \preceq  n^{-1/2} \log n$, so that Theorem~\ref{thm:main} yields the upper bound
\[
\Phi(n) = O\left(\frac{(\log \log n)^2}{\log n}\right);
\]
this bound is larger than the true order $\Phi(n)\asymp 1/\log n$ by a $(\log\log n)^2$ factor. Our inequalities are vacuous for the higher-dimensional lamplighter groups $\Z_2 \wr \Z^d$ when $d\geq 3$, in which case $\operatorname{TV}_n$ does not decay to zero, and are very weak when $d=2$ and $\operatorname{TV}_n$ decays very slowly. (In each case, the actual order of the isoperimetric profile is $(\log n)^{-1/d}$.) 
\end{example}

\medskip
\noindent \textbf{Other relations between total variation, entropy, isoperimetry, etc.} We now briefly overview the other relationships that are known between the various quantities we consider, referring the reader to e.g.\ \cite{zheng2022asymptotic,hutchcroft2024small} for further background.  First, it follows by standard arguments that
\[
\frac{1}{n}(D_n^*)^2 \preceq H_n^* \preceq D_n^*
\]
for any bounded degree graph: the lower bound is an immediate consequence of the Varopoulos-Carne inequality \cite{Varopoulos1985b, Carne1985} while the upper bound follows easily from the fact that a random variable supported on a set of size $n$ has entropy at most $\log n$. Relations between the displacement, entropy, and isoperimetric profile have typically been established via intermediate quantities including the random walk return probability $\sup_x P^n(x,x)$ and an associated quantity called the \emph{spectral profile}. The known relationships between return probabilities and the isoperimetric and spectral profiles yield, for example, that if $\alpha>0$, then
\[
\left(\Phi(n) \succeq (\log n)^{-\alpha} \right)\Rightarrow \left( - \log P^n(o,o) \preceq n^{1/(1+2\alpha)} \right) \Rightarrow  \left(\Phi(n) \succeq (\log n)^{-2\alpha} \right);
\]
see \cite{hutchcroft2024small} for further discussion of these inequalities and the question of whether they can both be sharp for Cayley graphs. On the other hand, using the fact that the lazy random walk transition matrix is positive semidefinite to bound off-diagonal transition probabilities by return probabilities, we can trivially bound the entropy and displacement in terms of the return probability decay
\[
D_n^*,H_n^* \succeq - \inf_{x\in V} \log P^n(x,x),
\]
where the displacement bound can be improved significantly if the graph has subexponential volume growth. For random walks on groups, Peres and Zheng \cite{peres2020groups} proved complementary inequalities between entropies and return probabilities that imply in particular that
\[
\left(- \log P^n(o,o) = O(n^\beta) \right) \Rightarrow \left( H_n^* = O(n^{\beta/(1-\beta)}) \right)
\]
for every $\beta>0$ (note, however, that this inequality is trivial when $\beta \geq 1/2$). See also \cite{MR1232845} for relations between the isoperimetric profile and \emph{volume growth} of Cayley graphs. The total variation profile has received much less attention than these other quantities, with the relation between the total variation and isoperimetric profiles being the main novelty of this paper. One important known relationship between the total variation profile and the entropy is the inequality
\[
\operatorname{TV}_n \preceq \sqrt{H_{n+1}^*-H_n^*},
\]
which holds for random walks on groups and, in a slightly different formulation, for unimodular random graphs. This inequality has been rediscovered multiple times by different authors \cite{MR3395463,erschler2010homomorphisms}, and plays an important role in Ozawa's functional-analytic proof of Gromov's theorem \cite{ozawa2018functional}. (It appears that using this inequality to bound $\operatorname{TV}_n$ and then applying \cref{thm:main} typically yields worse results than bounding the isoperimetric profile in terms of the entropy via intermediate estimates on the return probability and spectral profile.)


\subsection{Consequences for graphs of non-negative Ollivier-Ricci curvature}
\label{subsec:Ricci}

We now explain the applications of \cref{thm:main} to graphs of non-negative \emph{Ollivier-Ricci curvature} \cite{ollivier-ricci-markov},
 an analogue of the Ricci curvature that is defined for graphs and has recently attracted significant attention due to striking applications in the theory of Markov chain mixing \cite{munch-salez,eldan2017transport,salez2023cutoff,munch2023ollivier}.

\medskip

We begin by giving the relevant definitions.
    Let $(X,d)$ be a metric space, and let $\mu,\nu$ be two probability measures on $X$. The \emph{$L^1$-Wasserstein distance} between $\mu$ and $\nu$ is defined as
    \begin{align}
        \mathcal{W}_1(\mu,\nu)= \inf_{\pi \in \Pi(\mu,\nu)} \int_{X^2} d(x,y)d\pi (x,y),
    \end{align}
    where $\Pi(\mu,\nu)$ is the set of measures on $X\times X$ projecting to $\mu$ and $\nu$.
Let $G=(V,E)$ be a graph such that $V$ is a metric space when equipped with the graph distance $d_G$.
%
    The \textbf{Ollivier-Ricci curvature} of a graph $G$ at an edge $\{x,y\} \in E(G)$ is defined as
    \begin{align}
        \operatorname{Ric}(G)(x,y) = 1-\mathcal{W}_1(P_G(x,\cdot),P_G(y,\cdot)),
    \end{align}
    where $P_G(\cdot,\cdot)$ is the one-step transition matrix of the simple lazy random walk on $G$. We say that $G$ has \textbf{non-negative Ollivier-Ricci curvature} if $\operatorname{Ric}(G)(x,y) \geq 0$ for every $\{x,y\} \in E(G)$. Equivalently, $G$ has non-negative Ollivier-Ricci curvature if the lazy random walks $X$ and $Y$ started at any two vertices can be coupled in such a way that $(d(X_n,Y_n))_{n\geq 0}$ is a supermartingale. Prominent examples of graphs with non-negative curvature include Cayley graphs of abelian groups and, more generally, Cayley graphs taken with respect to conjugation-invariant generating sets (e.g.\ the Cayley graph of the symmetric group generated by the set of all transpositions).



This definition should be compared with the Ricci curvature of a Riemannian manifold, which measures the rate at which geodesics started from two nearby points (defined by parallel transport) coalesce or disperse:  for graphs, the Ollivier--Ricci curvature instead measures the rate at which two \emph{random walks}  coalesce or disperse under an optimal coupling. 

It has long been known that non-negative Ricci curvature strongly restricts the geometry of a Riemannian manifold. A notable example is the \emph{Bishop-Gromov inequality} \cite{petersen-rg}, which states that for a Riemannian manifold $M^n$ with $\operatorname{Ric} \geq 0$, the function
\begin{align}
    \phi_{p,q}(r) = \frac{\Vol(B_M(p,r))}{\Vol(B_{\mathbb{R}^n}(q,r))}
\end{align}
is non-increasing on $(0,\infty)$ for any $(p,q) \in M\times \mathbb{R}^n$. In particular, manifolds of non-negative Ricci curvature must have (sub)polynomial volume growth.
It is natural to ask whether non-negative Ollivier--Ricci curvature implies similarly strong geometric consequences for graphs, and to our knowledge, it remains open whether there exist non-negatively curved, bounded-degree graphs of superpolynomial growth.

Some important work has already been done in this direction. In \cite{salez2022sparse}, Salez proved that sparse, non-negatively curved graphs cannot be expanders, in the sense that these graphs must contain sets of small boundary-to-volume ratio when they are sufficiently large. 
 In \cite{munch-salez}, Münch and Salez proved a quantitative version of this theorem for finite graphs, stating that if $G=(V,E)$ is a finite graph of non-negative Ollivier-Ricci curvature and maximum degree at most $M\geq 2$, then its Cheeger constant must satisfy
 \begin{equation}
 \label{eq:Munch-Salez}
 \Phi:=\min\left\{\frac{|\partial W|}{\sum_{w\in W} \deg(w)} : \sum_{w\in W} \deg(w)\leq \frac{1}{2} \sum_{v\in V} \deg(v)\right\} \leq C \sqrt{\frac{M\log M}{\log |V|}}
 \end{equation}
 for some universal constant $C$. The dependence of this bound on the volume and the degree are essentially optimal in some regimes, since random Cayley graphs of logarithmic degree are expanders \cite{alon1994random}. The same authors also proved 
 diffusive lower bounds for the random walk and related estimates on mixing times under the same hypotheses.



Since the proof of \cite{munch-salez} relies on the analysis of global quantities such as the mixing time, 
it does not yield quantitative estimates on the isoperimetry of sets much smaller than $|V|$, and is not suitable for the analysis of infinite graphs.
Since their proof also establishes the total variation bound
 \[\operatorname{TV}_n \leq \sqrt{\frac{20 M }{(n+1)}}\]
 for all graphs of non-negative Ollivier--Ricci curvature of maximum degree at most $M$ \cite[Page 15]{munch-salez} (see also \cite[Proposition 16]{salez2022sparse}), our main theorem allows for an extension of the M\"unch-Salez estimate to all scales, and in particular for an extension of their estimate to infinite graphs.

\begin{corollary}
    Let $M<\infty$ and let $G$ be a graph with degrees bounded by $M$ and non-negative Ollivier-Ricci curvature. Then the isoperimetric profile $\Phi$ of $G$ satisfies
    \[
    \Phi(n) \leq C \sqrt{ \frac{M\log M}{\log n}}
    \]
    for every $n\geq 2$, where $C$ is a universal constant.
\end{corollary}

As before, the example of random abelian Cayley graphs \cite{alon1994random} shows that this estimate is essentially optimal (up to the $\sqrt{\log M}$ factor) in the high-degree regime. On the other hand, it is plausible that much stronger estimates hold when the degrees are fixed and $n\to\infty$.

\section{Upper tail estimates for the random walk displacement and information}
\label{sec:upper_tail}
The goal of this section is to prove the following theorem, which states that the random walk displacement $d_G(X_0,X_n)$ and information $-\log P^n(X_0,X_n)$ are exponentially unlikely to exceed their expected values by a large multiplicative factor.

\begin{theorem}
\label{thm:upper_tail}
Let $G=(V,E)$ be a graph with degrees bounded by $M$. There exists a posiitve constant $c=c(M)$ such that
\[
\mathbf{P}_x( d_G(X_0,X_n) \geq \lambda D_n^*) \leq e^{-c\lambda} \; \text{ and } \; \mathbf{P}_x\left( -\log P^n(X_0,X_n) \geq \lambda (H_n^* + \log n) \right) \leq e^{-c\lambda}
\]
for every $x\in V$, $n\geq 2$, and $\lambda \geq 1$.
\end{theorem}

We will deduce both statements from the following more general statement about arbitrary metrics on the graph. The statement about the random walk information will be deduced by applying this lemma to the \emph{massive Green metric}, defined below.

\begin{prop}
\label{prop:Green_metric_upper_tail_strong}
Let $G=(V,E)$ be a locally finite graph with degrees bounded by $M\geq 2$ and let $d$ be a metric on $V$. There exists a universal positive constant $c$ such that
\[\mathbf{P}_x\left(\max_{0\leq m \leq n} d(X_0,X_m) \geq \lambda \sup_y \max_{0\leq m \leq n} \mathbf{E}_y \left[d(X_0,X_m)\right]  \right) \leq \exp\left[1-\frac{c \lambda}{M^2\log M}\right]\]
for every $x\in V$,  $\lambda \geq 1$, and integer $n\geq 1$. 
\end{prop}

Note that this lemma holds vacuously if the supremum inside the probability is infinite. We begin by proving the following slight variation on this proposition.

\begin{lemma}
\label{lem:Green_metric_upper_tail}
Let $G=(V,E)$ be a locally finite graph with degrees bounded by $M\geq 2$ and let $d$ be a metric on $V$.
Then
\[\mathbf{P}_x\left(\max_{0\leq m \leq n} d(X_0,X_m) \geq \lambda \sup_y \mathbf{E}_y \left[\max_{0\leq m \leq n} d(X_0,X_m)\right] \right) \leq \exp\left[-\left\lfloor \frac{\lambda}{M+e}\right\rfloor \right]\]
for every $x\in V$,  $\lambda \geq 1$, and integer $n\geq 1$. 
\end{lemma}


\begin{proof}[Proof of \cref{lem:Green_metric_upper_tail}]
 Set $A=\sup_{u\sim v} d(u,v)$, let $k\geq 1$ be an integer, and let $R >0$ be a parameter. Let the stopping times $\tau_0,\tau_1,\ldots$ be defined inductively by $\tau_0=0$ and 
\[
\tau_i = \inf\{ t\geq \tau_{i-1} : d(X_t,X_{\tau_{i-1}}) \geq R\},
\]
 setting $\tau_i=\infty$ if the set being infimized over is empty. At the time $\tau_i$, we must have $R\leq d(X_{\tau_i},X_{\tau_{i-1}})\leq R+A$ by definition of the constant $A$. As such, we have that
 \[
 \mathbf{P}_x\left( \max_{0\leq m \leq n} d(X_0,X_m) \geq k (R+A)\right) \leq \mathbf{P}_x(\tau_1,\ldots,\tau_k \leq n) \leq \left[\sup_y \mathbf{P}_y(\tau_1 \leq n)\right]^k,
 \]
 where the final inequality follows by the strong Markov property. Now, if we take the radius $R$ to be $R=e\sup_y \mathbf{E}_y [\max_{0\leq m \leq n} d(X_0,X_m)]$, then it follows by Markov's inequality that $\sup_y \mathbf{P}_y(\tau_1 \leq n)\leq e^{-1}$ and hence that
 \[
  \mathbf{P}_x\left( \max_{0\leq m \leq n} d(X_0,X_m) \geq k \left(e\sup_y \mathbf{E}_y \left[\max_{0\leq m \leq n} d(X_0,X_m)\right]+A\right)\right) \leq e^{-k}
 \]
 for every integer $k\geq 1$. Note that this form of the inequality does not require the degrees to be bounded. This inequality is easily seen to yield the claim since, by definition, $A=\sup_{x\sim y} d(u,v) \leq 2M\sup_y \mathbf{E}_y [\max_{0\leq m \leq n} d(X_0,X_m)]$ for every $n\geq 1$. 
\end{proof}

To deduce \cref{prop:Green_metric_upper_tail_strong} from \cref{lem:Green_metric_upper_tail}, we need to show that the max can be taken out of the expectation $\sup_x\mathbf{E}_x[\max_{0\leq m \leq n} d(X_0,X_m)]$
without changing its order. This will be deduced from the following simple lemma.

\begin{lemma}
\label{lem:triangle}
Let $G=(V,E)$ be a locally finite graph and let $d$ be a metric on $V$. For each $n\geq 1$ and $r\geq 1$, we have that
\[
\sup_x \max_{0\leq m \leq n} \mathbf{P}_x\bigl(d(X_0,X_m)\geq r\bigr) \geq \frac{1}{2} \sup_x\mathbf{P}_x\Bigl(\max_{0\leq m \leq n} d(X_0,X_m) \geq 2r\Bigr).
\]
\end{lemma}

\begin{proof}[Proof of \cref{lem:triangle}]
Let $r\geq 1$ and let $\tau$ be the first time that $d(X_0,X_\tau)\geq 2r$, setting $\tau=\infty$ if this never occurs. We have by the strong Markov property that
\begin{align*}
\mathbf{P}_x(d(X_0,X_m) \geq r) &\geq \mathbf{P}_x(\max_{0\leq s \leq m} d(X_0,X_s) \geq 2r) - \mathbf{E}_x\left[ \mathbf{P}_{X_\tau} (d(X'_0,X'_{m-\tau}\geq r) \right]
\\
&\geq \mathbf{P}_x(\max_{0\leq s \leq m} d(X_0,X_s) \geq 2r) - \sup_{y} \sup_{0\leq k \leq m} \mathbf{P}_y(d(X_0,X_{k})\geq r)
\end{align*}
for every $x\in V$, $r\geq 1$, and $0\leq m \leq n$, where the inner probability in the last term on the first line is with respect to a new random walk started at $X_\tau$. 
\end{proof}

\begin{corollary}
\label{cor:expectation_and_median}
    Let $G=(V,E)$ be a locally finite graph with degrees bounded by $M\geq 2$ and let $d$ be a metric on $V$. We have that
    \begin{equation*}
    \sup_x \sup_{0\leq m \leq n} \mathbf{P}_x \left(d(X_0,X_m)\geq \frac{1}{16} \sup_y \mathbf{E}_y \left[\max_{0\leq \ell \leq n} d(X_0,X_\ell)\right] \right) \geq \frac{c}{M\log M} 
    \end{equation*}
    and hence that
    \begin{equation*}
    \sup_x \sup_{0\leq m \leq n} \mathbf{E}_x \left[d(X_0,X_m)\right] \geq \frac{c}{16M\log M} \sup_y \mathbf{E}_y \left[\max_{0\leq m \leq n} d(X_0,X_m)\right]
    \end{equation*}
    for every $n\geq 1$, where $c$ is a positive universal constant.
\end{corollary}

\begin{proof}[Proof of \cref{cor:expectation_and_median}]
For each $n\geq 1$, let $\mathbf{M}_n = \sup_y \mathbf{E}_y\left[\max_{0\leq m \leq n} d(X_0,X_m)\right]$.
Fix $n\geq 1$ and let $x\in V$ be such that
\[
\mathbf{E}_x\left[\max_{0\leq m \leq n} d(X_0,X_m)\right] \geq \frac{1}{2} \mathbf{M}_n.
\]
For each $\lambda \geq 1$, we have by \cref{lem:Green_metric_upper_tail} that
\begin{multline*}
\mathbf{E}_x\left[\max_{0\leq m \leq n} d(X_0,X_m) \,\mathbbm{1}\!\left( \max_{0\leq m \leq n} d(X_0,X_m) \geq \lambda \mathbf{M}_n\right) \right] \leq \mathbf{M}_n \int_\lambda^\infty \exp\left[ -\left\lfloor \frac{t}{M+e}\right\rfloor  \right] \dif t
\\
\leq \frac{e(M+e)}{e-1}\mathbf{M_n} \exp\left[ -\left\lfloor \frac{\lambda}{M+e}\right\rfloor  \right].
\end{multline*}
If $\lambda \geq \lambda_0:= (M+e)\lceil \log \frac{4e(M+e)}{e-1}\rceil$, then the left hand side is at most $\frac{1}{4}\mathbf{M}_n$,  so that
\[\mathbf{E}_x\left[\max_{0\leq m \leq n} d(X_0,X_m) \, \mathbbm{1}\!\left( \max_{0\leq m \leq n} d(X_0,X_m) \leq \lambda_0 \mathbf{M}_n\right) \right] \geq \frac{1}{4}\mathbf{M_n}. \]
It follows from this and Markov's inequality (applied to the difference between the random variable and $\lambda_0 M$) that
\[\mathbf{P}_x\left(\max_{0\leq m \leq n} d(X_0,X_m) \geq \frac{1}{8} \mathbf{M}_n\right) \geq 1-\frac{\lambda_0-1/4}{\lambda_0-1/8} \geq \frac{1}{8\lambda_0}, \]
where $c$ is a positive universal constant. Applying \cref{lem:triangle}, we deduce that
\[
\sup_x \max_{0\leq m \leq n} \mathbf{P}_x\left( d(X_0,X_m) \geq \frac{1}{16} \mathbf{M}_n\right) \geq 1-\frac{\lambda_0-1/4}{\lambda_0-1/8} \geq \frac{1}{16\lambda_0} \geq \frac{c}{M\log M},\]
where $c$ is a positive universal constant. The claim concerning the expectation follows trivially.
\end{proof}

\begin{proof}[Proof of \cref{prop:Green_metric_upper_tail_strong}]
This follows immediately from \cref{lem:Green_metric_upper_tail} and \cref{cor:expectation_and_median}.
\end{proof}

We now prove \cref{thm:upper_tail}. The inequality concerning the displacement is an immediate corollary of \cref{prop:Green_metric_upper_tail_strong}, while the inequality concerning the information will require some more definitions.
For each $t>0$, we write $\mathbf{G}_t(u,v)$ for the probability that the lazy random walk started at $u$ hits $v$ before a geometric random time of mean $t$. It follows by the strong Markov property that $\mathbf{G}_t(u,v) \geq \mathbf{G}_t(u,w)\mathbf{G}_t(w,v)$ for every $t>0$ and $u,v,w\in V$, so that $-\log \mathbf{G}_t(u,v)$ defines a metric on $G$ for each $t>0$, which we call the \textbf{massive Green metric}. (When $G$ is an infinite transient graph, we may set $t=\infty$ to recover the usual Green metric on the graph.) We can compare the massive Green metric with the random walk information using the following elementary inequality, which uses a similar argument to that of \cite[Proposition 3.2]{halberstam2023most}.

\begin{lemma}
\label{lem:information_to_Green}
Let $G=(V,E)$ be a locally finite graph. The inequalities
\[
 \E \left[\frac{\mathbf{G}_t(X_0,X_n)}{P^n(X_0,X_n)}\right] \leq t+1 \qquad \text{ and } \qquad \frac{\mathbf{G}_t(X_0,X_n)}{P^n(X_0,X_n)} \geq \left(1-\frac{1}{t}\right)^n 
\]
hold for every $x\in V$, $n\geq 0$, and $t>0$.
\end{lemma}

\begin{proof}[Proof of \cref{lem:information_to_Green}]
The lower bound follows trivially from the fact that $\mathbf{G}_t(u,v) \geq (1-1/t)^n P^n(u,v)$. For each $m\geq 0$, we also have that 
\[
\mathbf{E}_x \left[\frac{P^m(X_0,X_n)}{P^n(X_0,X_n)}\right] = \sum_{y\in V} \frac{P^m(x,y)}{P^n(x,y)}P^n(x,y) \mathbbm{1}(P^n(x,y)\neq 0) = \mathbf{P}_x(P^m(X_0,X_n)>0) \leq 1,
\]
and the claimed upper bound follows by summing this estimate over $m$, weighted by the probability that our mean-$t$ geometric random variable is at least $m$.
\end{proof}

\begin{proof}[Proof of \cref{thm:upper_tail}]
The bound concerning the displacement is an immediate consequence of \cref{prop:Green_metric_upper_tail_strong}. The bound concerning the information follows from \cref{prop:Green_metric_upper_tail_strong} applied to the massive Green metric $d(x,y)=-\log \mathbf{G}_n(x,y)$. Indeed, it follows from \cref{lem:information_to_Green} and Jensen's inequality that
\[
\mathbf{E}_x[-\log P^m(X_0,X_m)] - \log (1+n) \leq \mathbf{E}_x[-\log \mathbf{G}_n(X_0,X_m)] \leq \mathbf{E}_x[-\log P^m(X_0,X_m)] + C
\]
for a universal constant $C$, and that
\begin{equation}
\mathbf{P}_x\left(-\log P^m(X_0,X_m) \geq - \log \mathbf{G}_n(X_0,X_m) + \mu \right) \leq \mathbf{P}_x\left( \frac{G_n(X_0,X_m)}{P^m(X_0,X_m)} \geq e^\mu \right) \leq (n+1) e^{-\mu}
\label{eq:info_larger_than_green_tail}
\end{equation}
for every $\mu>0$.
Thus, we have by a union bound that
\begin{align*}
&\mathbf{P}_x\left(-\log P^m(X_0,X_m) \geq -\lambda \left(H_n^*+ \log (n+1) \right)+\mu\right)
\\ 
&\hspace{3.5cm}\leq \mathbf{P}_x\left(-\log P^m(X_0,X_m) \geq -\lambda \left(\sup_y \max_{0\leq m \leq n} \mathbf{E}_y[-\log \mathbf{G}_n(X_0,X_m)] \right)+\mu\right)
\\&\hspace{3.5cm}\leq \mathbf{P}_x\left(-\log \mathbf{G}_n(X_0,X_m) \geq -\lambda \left(\sup_y \max_{0\leq m \leq n} \mathbf{E}_y[-\log \mathbf{G}_n(X_0,X_m)] \right)\right)
\\
&\hspace{7cm}+
\mathbf{P}_x\left(-\log P^m(X_0,X_m) \geq - \log \mathbf{G}_n(X_0,X_m) + \mu \right)
\\&\hspace{3.5cm}\leq e^{1-c\lambda} + (n+1)e^{-\mu}
\end{align*}
for some constant $c(M)>0$, where the final inequality follows from \cref{prop:Green_metric_upper_tail_strong} and \eqref{eq:info_larger_than_green_tail}. This is easily seen to imply the claim.
\end{proof}

\section{Construction of sets of small boundary}
 \label{construction-of-subset}
We now describe the construction of sets of small boundary-to-volume ratio used to prove \cref{thm:main}. The proof will rely on the following ``simultaneous near-optimal coupling'' theorem of Angel and Spinka \cite{angel2019pairwise}.

\begin{theorem}[{\!\!\cite{angel2019pairwise}, Theorem 2}] \label{angel-spinka}
    Let $\mathcal{S}$ be a countable collection of probability measures, all absolutely continuous with respect to a common $\sigma$-finite measure. Then there exists a coupling $\P$ of all the measures in $\mathcal{S}$ such that, for any two measures $\P_X,\P_Y\in \mathcal{S}$, the two random variables $X$ and $Y$ with distributions $\P_X$ and $\P_Y$ in the coupling satisfy
    \begin{align}
        \P(X\neq Y) \leq \frac{2||\P_X-\P_Y||_{\tv}}{1+||\P_X-\P_Y||_{\tv}} \leq 2||\P_X-\P_Y||_{\tv}.
    \end{align}
\end{theorem}

We will also need the following elementary lemma.

\begin{lemma}
\label{lem:TV_conditioning}
Let $\P$ and $\Q$ be probability measures on the same measurable space, let $A$ and $B$ be events of positive probability for $\P$ and $\Q$ respectively, and let $\P|_A$ and $\Q|_B$ be the associated conditional measures. Then 
$\|\P|_A-\Q|_A\|_{\operatorname{TV}} \leq \|\P-\Q\|_{\operatorname{TV}} + \P(A^c)+\Q(B^c)$.
\end{lemma}

\begin{proof}[Proof of \cref{lem:TV_conditioning}]
It is easily verified from the definition of total variation that $\|\P-\P|_A\|_{\operatorname{TV}} \leq 1-\P(A)$, and the claim follows by the triangle inequality.
\end{proof}

Let $G=(V,E)$ be a connected, locally finite graph.
For each $x\in V$ and $n\geq 1$, let $\mathbf{P}_{x,n}$ be the law of the $n$th location $X_{x,n}$ of the lazy random walk  started at $x$. 
For each $x\in V$, $n\geq 1$, and $\lambda>0$, let $\mathscr{E}_{x,n,\lambda}$ denote the event that $d(x,X_{x,n})\leq \lambda D_n^*$ and $-\log P^n(x,X_{x,n}) \leq \lambda (H_n^*+\log n)=:\log N_{n,\lambda}$, where we define $N_{n,\lambda}$ so that this equality is satisfied. It follows from \cref{thm:upper_tail} that
\begin{equation}
\label{eq:upper_tail_restated}
\sup_x\P(\mathscr{E}_{x,n,\lambda}^c)  \leq 2e^{-c\lambda},
\end{equation}
where $c=c(M)$ is a positive constant depending only on $M$.
Applying Theorem~\ref{angel-spinka} to the collection of conditional distributions $\mathcal{S}_{n,\lambda}=\{\P_{n,x}|_{\mathscr{E}_{x,n,\lambda}} :x\in V\}$, we see that there is a coupling $\P$ of all these conditioned random walks such that 
\begin{align}
\label{eq:TV_tilde_def}
    \widetilde{\operatorname{TV}}_{n,\lambda}:=\sup_{x\sim y} \P(X_{x,n}\neq X_{y,n}) \leq 2\operatorname{TV}_n + 2 \sup_x\P(\mathscr{E}_{x,n,\lambda}^c) 
\end{align}
(The coupling also gives some control of non-neighboring vertices, but we will not use this.) We use this coupled collection of random walks to define a random equivalence relation $\Pi_{n,\lambda}$, where we set $x$ and $y$ to belong to the same class of $\Pi_{n,\lambda}$ if $X_{x,n}=X_{y,n}$.
We will denote by $[x]_{n,\lambda}$ the set of vertices $y$ that are related to $x$ under $\Pi_{n,\lambda}$. Thus, all the elements $y$ of $[x]_{n,\lambda}$ have a common value of $X_{y,n}$, which we denote by $\sigma_{n,\lambda}(x)=\sigma_{n,\lambda}([x]_{n,\lambda})$.
 Observe that every element of $[x]_{n,\lambda}$ has distance at most $\lambda D_n^*$ from $\sigma_{n,\lambda}(x)$, so that $[x]_{n,\lambda}$ has diameter at most $2\lambda D_n^*$. Moreover, since every element $y$ of $[x]_{n,\lambda}$ has $P^n(\sigma_{n,\lambda}(x),y) \geq N_{n,\lambda}^{-1}$ and $\sum_w P^n(\sigma_{n,\lambda}(x),1)=1$, we must have that $[x]_{n,\lambda}$ contains at most $N_{n,\lambda}$ points. We also have trivially that $[x]_{n,\lambda}$ has diameter at most $2n$ and volume at most $M^n$. To conclude the proof of the theorem, it therefore suffices to prove that at least one of these cells has boundary-to-volume ratio of order at most $\widetilde{\operatorname{TV}}_{n,\lambda}$ with positive probability. Taking $\lambda$ of order $-\log \operatorname{TV}_n$ will then ensure that this is of the same order as $\operatorname{TV}_n$. 

\medskip


The following key lemma, which is based implicitly on the \emph{mass-transport principle}, allows us to bound the expected boundary-to-volume ratio of the cell at a uniform random point. For infinite graphs, we make sense of this by restricting to a large finite set.

\begin{lemma}
\label{lem:finite_MTP}
Suppose that $G=(V,E)$ is a locally finite graph and let $F$ be a finite subset of $V$. Then
\[
\frac{1}{\deg(F)}\sum_{x\in F}\deg(x)\E\left[ \frac{\partial ([x]_{n,\lambda}\cap F)}{\deg([x]_{n,\lambda} \cap F)}\right] \leq \widetilde{\operatorname{TV}}_{n,\lambda} + \frac{|\partial F|}{\deg(F)}
\]
for every $n\geq 1$ and $\lambda>0$.
\end{lemma}

When $G$ is finite, we can apply this lemma with $F=V$ so that the second term vanishes.
(An infinite-volume version of this lemma also holds for unimodular random rooted graphs.)

\begin{proof}[Proof of \cref{lem:finite_MTP}]
Lightening notation by writing $[x]=[x]_{n,\lambda} \cap F$, we can write
\begin{align*}
\sum_{x\in F}\deg(x)\E\left[ \frac{\partial [x]}{\deg([x])}\right] &= 
\E\left[ \sum_{x\in F} \deg(x)\sum_{y\in [x]} \frac{1}{\deg([x])}  \sum_{e^-=y} \mathbbm{1}(e^+ \notin [x])\right]
\\&=
\E\left[ \sum_{y\in V} \sum_{x\in [y]} \frac{\deg(x)}{\deg([y])} \sum_{e^-=y} \mathbbm{1}(e^+ \notin [x])\right]
\\ &= \sum_{y\in V} \sum_{e^-=y} \P(e^+ \notin [y]) \leq \sum_{y\in F} \sum_{e^-=y} [\mathbbm{1}(e^+\notin F) + \P(e^+\notin [y]_{n,\lambda})].
\end{align*}
The sum $\sum_{y\in V} \sum_{e^-=y} \mathbbm{1}(e^+\notin F)$ is equal to $|\partial F|$. Meanwhile, the definitions ensure that $\P(e^+ \notin [y]_{n,\lambda}) \leq \widetilde{\operatorname{TV}}_{n,\lambda}$ for every $y\in F$ and oriented edge $e$ emanating from $y$, and the claim follows by dividing both sides by $\deg(F)$.
\end{proof}

\begin{proof}[Proof of \cref{thm:main}]
 If we take $\lambda = C \log \operatorname{TV}_n$ for an appropriately large constant $C=C(M)$, then it follows from \eqref{eq:upper_tail_restated} that $\widetilde{\operatorname{TV}}_{n,\lambda}\leq 2 \operatorname{TV}_n$. 
Moreover, since $G$ is amenable, there exists a finite set of vertices $F$ such that $|\partial F|/\deg(F)\leq 2 \operatorname{TV}_n$. (When $G$ is finite we can take $F=V$.)
 Thus, by \cref{lem:finite_MTP}, if $x$ is a uniform random vertex of $F$, then the expected boundary-to-volume ratio of the intersection $[x]_{n,\lambda}\cap F$ in the construction above is at most $4\operatorname{TV}_n$; this trivially implies that set satisfying the prescribed geometric constraints must exist. \qedhere
\end{proof}

\subsection*{Acknowledgements} 
This paper is the outcome of a summer undergraduate research project in 2024 as part of Caltech's WAVE program, where I.M.L.\ was mentored by T.H.\ and Antoine Song;
we thank Antoine Song for many valuable discussions throughout the course of the summer. I.M.L.\ was supported by a Carl F.\ Braun research fellowship, and thanks Caltech for its hospitality this summer. T.H.\ is supported by NSF grant DMS-2246494 and a Packard Fellowship in Science and Engineering. The idea for this project emerged during Stochastic Processes and Related Fields 2023 at RIMS, Kyoto, and T.H.\ thanks Justin Salez for inspiring conversations at that time.

\footnotesize{
\bibliographystyle{abbrv}
\bibliography{big_bib_file}

\begin{thebibliography}{10}

\bibitem{AL07}
D.~Aldous and R.~Lyons.
\newblock Processes on unimodular random networks.
\newblock {\em Electron. J. Probab.}, 12:no. 54, 1454--1508, 2007.

\bibitem{alon1994random}
N.~Alon and Y.~Roichman.
\newblock Random {C}ayley graphs and expanders.
\newblock {\em Random Structures \& Algorithms}, 5(2):271--284, 1994.

\bibitem{angel2019pairwise}
O.~Angel and Y.~Spinka.
\newblock Pairwise optimal coupling of multiple random variables.
\newblock {\em arXiv preprint arXiv:1903.00632}, 2019.

\bibitem{MR3395463}
I.~Benjamini, H.~Duminil-Copin, G.~Kozma, and A.~Yadin.
\newblock Disorder, entropy and harmonic functions.
\newblock {\em Ann. Probab.}, 43(5):2332--2373, 2015.

\bibitem{Carne1985}
T.~K. Carne.
\newblock A transmutation formula for {M}arkov chains.
\newblock {\em Bull. Sci. Math. (2)}, 109(4):399--405, 1985.

\bibitem{MR1232845}
T.~Coulhon and L.~Saloff-Coste.
\newblock Isop\'{e}rim\'{e}trie pour les groupes et les vari\'{e}t\'{e}s.
\newblock {\em Rev. Mat. Iberoamericana}, 9(2):293--314, 1993.

\bibitem{eldan2017transport}
R.~Eldan, J.~R. Lee, and J.~Lehec.
\newblock Transport-entropy inequalities and curvature in discrete-space {M}arkov chains.
\newblock {\em A Journey Through Discrete Mathematics: A Tribute to Ji{\v{r}}{\'\i} Matou{\v{s}}ek}, pages 391--406, 2017.

\bibitem{erschler2010homomorphisms}
A.~Erschler and A.~Karlsson.
\newblock Homomorphisms to $\mathbb{R}$ constructed from random walks.
\newblock In {\em Annales de l'Institut Fourier}, volume~60, pages 2095--2113, 2010.

\bibitem{halberstam2023most}
N.~Halberstam and T.~Hutchcroft.
\newblock Most transient random walks have infinitely many cut times.
\newblock {\em The Annals of Probability}, 51(5):1932--1962, 2023.

\bibitem{hutchcroft2024small}
T.~Hutchcroft.
\newblock Small-ball estimates for random walks on groups.
\newblock {\em arXiv preprint arXiv:2406.17587}, 2024.

\bibitem{KumagaiBook}
T.~Kumagai.
\newblock {\em Random walks on disordered media and their scaling limits}, volume 2101 of {\em Lecture Notes in Mathematics}.
\newblock Springer, Cham, 2014.
\newblock Lecture notes from the 40th Probability Summer School held in Saint-Flour, 2010, {\'E}cole d'{\'E}t{\'e} de Probabilit{\'e}s de Saint-Flour. [Saint-Flour Probability Summer School].

\bibitem{LP:book}
R.~Lyons and Y.~Peres.
\newblock {\em Probability on Trees and Networks}, volume~42 of {\em Cambridge Series in Statistical and Probabilistic Mathematics}.
\newblock Cambridge University Press, New York, 2016.

\bibitem{munch2023ollivier}
F.~M{\"u}nch.
\newblock Ollivier curvature, isoperimetry, concentration, and log-{S}obolev inequalitiy.
\newblock {\em arXiv preprint arXiv:2309.06493}, 2023.

\bibitem{munch-salez}
F.~Münch and J.~Salez.
\newblock Mixing time and expansion of non-negatively curved markov chains.
\newblock {\em J. l’École polytechnique}, 10:575--590, 2023.

\bibitem{ollivier-ricci-markov}
Y.~Ollivier.
\newblock Ricci curvature of markov chains on metric spaces.
\newblock {\em J. Funct. Anal.}, 256(3):810--864, 2009.

\bibitem{ozawa2018functional}
N.~Ozawa.
\newblock A functional analysis proof of {G}romov’s polynomial growth theorem.
\newblock {\em Ann. Sci. {\'E}c. Norm. Sup{\'e}r.(4)}, 51(3):549--556, 2018.

\bibitem{peres2020groups}
Y.~Peres and T.~Zheng.
\newblock On groups, slow heat kernel decay yields {L}iouville property and sharp entropy bounds.
\newblock {\em International Mathematics Research Notices}, 2020(3):722--750, 2020.

\bibitem{petersen-rg}
P.~Petersen.
\newblock {\em Riemannian Geometry}.
\newblock Springer, 2016.

\bibitem{salez2022sparse}
J.~Salez.
\newblock Sparse expanders have negative curvature.
\newblock {\em Geometric and Functional Analysis}, 32(6):1486--1513, 2022.

\bibitem{salez2023cutoff}
J.~Salez.
\newblock Cutoff for non-negatively curved {M}arkov chains.
\newblock {\em Journal of the European Mathematical Society}, 26(11):4375--4392, 2023.

\bibitem{Varopoulos1985b}
N.~T. Varopoulos.
\newblock Long range estimates for {M}arkov chains.
\newblock {\em Bull. Sci. Math. (2)}, 109(3):225--252, 1985.

\bibitem{zheng2022asymptotic}
T.~Zheng.
\newblock Asymptotic behaviors of random walks on countable groups.
\newblock In {\em Proceedings of the International Congress of Mathematicians}, page~4, 2022.

\end{thebibliography}
}

\end{document}